\documentclass[11pt]{article}

\usepackage{amsfonts}
\usepackage{amscd}
\usepackage{amsmath}
\usepackage{epsfig}
\usepackage{amsthm}
\usepackage{amssymb}
\usepackage{amsbsy}
\usepackage{latexsym}
\usepackage{eucal}
\usepackage{mathrsfs}

\newtheorem{theorem}{Theorem}

\newtheorem{lemma}{Lemma}

\def\be{\begin{equation}}
\def\ee{\end{equation}}
\def\ben{\begin{displaymath}}
\def\een{\end{displaymath}}
\def\baa{\begin{eqnarray}}
\def\eaa{\end{eqnarray}}

\def\ba{\begin{array}}
\def\ea{\end{array}}
\renewcommand{\leq}{\leqslant}
\renewcommand{\geq}{\geqslant}

\newcommand{\supp}{\operatorname{supp}}

\newcommand{\Tr}{\operatorname{Tr}}
\newcommand{\Res}{\operatorname{Res}}
\newcommand{\Vol}{\operatorname{Vol}}

\newcommand{\Id}{\operatorname{Id}}
\newcommand{\Det}{\operatorname{Det}}

\makeatletter
\@addtoreset{equation}{section}
\makeatother

\begin{document}
\title {On Determinants of Laplacians on Compact Riemann Surfaces Equipped with Pullbacks of  Conical Metrics by Meromorphic Functions}

\author{Victor Kalvin \footnote{{\bf E-mail:  vkalvin@gmail.com, victor.kalvin@concordia.ca}}}
\date{Dec 12, 2017}
\maketitle

\vskip0.5cm
\begin{center}
Department of Mathematics and Statistics, Concordia
University, 1455 de Maisonneuve Blvd. West, Montreal, Quebec, H3G
1M8 Canada \end{center}

\vskip2cm
{\bf Abstract.}   Let  $\mathsf m$ be any  conical (or smooth) metric of finite volume on the Riemann sphere $\Bbb CP^1$. On a compact Riemann surface $X$ of genus $g$ consider a meromorphic funciton $f: X\to {\Bbb  C}P^1$   such that all poles and critical points of $f$ are simple and no critical value of $f$ coincides with a conical singularity of $\mathsf m$ or $\{\infty\}$.  The pullback $f^*\mathsf m$ of $\mathsf m$ under $f$ has conical singularities of  angles $4\pi$ at the  critical points of $f$ and  other conical singularities that are the preimages of those of $\mathsf m$.    We study the $\zeta$-regularized determinant $\Det' \Delta_F$ of the (Friedrichs extension of) Laplace-Beltrami  operator on $(X,f^*\mathsf m)$ as a functional on the moduli space of pairs $(X, f)$ and  obtain an explicit formula for  $\Det' \Delta_F$. 
\vskip2cm

\section{Introduction}\label{intro}

The problem of explicit evaluation of the determinants of Laplacians on Riemann surfaces has received considerable attention. 
For smooth metrics the determinants were thoroughly studied, see e.g.~\cite{DPh, Efrat,Sarnak, OPS}.  Over the past  decade significant progress was also achieved  for flat conical metrics, see e.g.~\cite{Au-Sal,HK,KKH-Communications,HKK,Khuri,KK-DG, ProcAMS,HK2}.  The problem of explicit evaluation of the determinants for conical metrics of constant positive curvature has also attracted some attention~\cite{Sp,Klevtsov,KKIMRN}. In this paper we derive an explicit formula for the determinant of Laplacian corresponding to the pullback of a finite volume conical metric  on the Riemann sphere by a meromorphic function.  In particular, the corresponding flat, constant positive curvature, and hyperbolic  conical metrics  are  included into consideration.

Let $X$ be a compact Riemann surface $X$ of genus $g$.  Following~\cite{Troyanov Polar Coordinates}, we say that $\mu$ is a  conical metric on $X$  if for any point $P\in X$ there exist a neighbourhood $U$ of $P$, a local parameter $x\in \Bbb C$, and a real-valued function $\varphi\in L^1(U)$ 
  such that $x(P)=0$, $\mu=e^{2\varphi}|x|^{2\beta}|dx|^2$ in $U$ with some $\beta>-1$, and $\partial_x\partial_{\bar x} \varphi\in L^1(U)$. If $\beta=0$, then the point $P$ is regular. If $\beta\neq 0$, then $P$ is a conical singularity of total angle $2\pi(\beta+1)$. A function $K:X\to \Bbb R$ defined by
 $$
K=e^{-2\varphi}|x|^{-2\beta}(-4\partial_x\partial_{\bar x} \varphi)
$$    
 is the curvature of $\mu$ in the neighbourhood $U$ ($K$ does not depend on the choice of $x$).

  Consider a conical metric $\mathsf m =\rho(z,\bar z)|dz|^2$ on the Riemann sphere ${\Bbb  C}P^1$ with conformal factor  
$$
\rho(z,\bar z)=e^{2u(z,\bar z)}\prod_{j=1}^n |z-p_j|^{2\beta_j},
$$ where $n\geq 0$, $\beta_j\in (-1,0)\cup(0,\infty)$ and $u\in C({\Bbb  C}P^1)$ is  a real-valued function.  At $z=p_j$ the metric $\mathsf m$ has a conical singularity of total angle $2\pi(\beta_j+1)$. We define the curvature  $K(z,\bar z)$ of $\mathsf m$ by the formula 
\begin{equation*}
K=e^{-2u}\prod_{j=1}^n |z-p_j|^{-2\beta_j}(-4 \partial_z\partial_{\bar z} u)
\end{equation*}
and  assume that   $K\in C^\infty(\Bbb CP^1)$ and $\Vol_{\mathsf m}(\Bbb C P^1)<\infty$.

 Consider  a meromorphic function $f: X\to {\Bbb  C}P^1$  of degree $N$  with simple poles and simple critical points. Assume that the critical values of $f$ are finite and do not coincide with conical singularities of $\mathsf m$.  Then the pullback $f^*\mathsf m$ of $\mathsf m$ by $f$ is a conical metric with $M=2N+2g-2$ conical singularities of  angles $4\pi$ at the critical points of $f$ and $N\times n$ conical singularities of angles $2\pi(\beta_j+1)$ at the preimages $f^{-1}(p_j)$  of the conical points $z=p_j$ of $\mathsf m$. It is easy to see that the curvature of $f^*\mathsf m$ is smooth and at  $P\in X$  it coincides with the curvature of  $\mathsf m$ at $z=f(P)$.  In particular, if $\mathsf m$ is a conical metric of constant curvature on $\Bbb CP^1$, then  $f^*\mathsf m$ is a conical metric of the same constant curvature on $X$.

We study the $\zeta$-regularized determinant $\Det' \Delta_F$ of the (Friedrichs extension of) Laplace-Beltrami  operator on $(X,f^*\mathsf m)$ as a functional on the Hurwitz moduli space $H_{g,N}(1,\dots,1)$ of pairs $(X, f)$. The space $H_{g,N}(1,\dots,1)$  (here $1$ is repeated $N$ times, i.e.  there are exactly $N$ distinct preimages $f^{-1}(\infty)$ of the point $\{\infty\}\in\Bbb CP^1$)  is a connected complex manifold locally  coordinatized by the critical values $z_1,\dots,z_M$ of the function $f$. The main result of this paper is the following  explicit formula 
\begin{equation}\label{teorema}
{\Det}'\Delta_F=C\,{\det}\Im {\mathbb B}\, |\tau|^2 \prod_{k=1}^M\sqrt[8]{\rho(z_k,\bar z_k)}.
\end{equation}
Here $C$ is a constant independent of the point $(X,f)$  of $H_{g,N}(1,\dots,1)$,  $ {\mathbb B}$ is the matrix of $b-$periods of the Riemann surface $X$ (for genus zero surface $X$ the factor ${\det}\Im {\mathbb B}$ should be omitted), $\tau$ stands for the Bergman tau-function on $H_{g,N}(1,\dots,1)$, and $\rho(z,\bar z)$ is the conformal factor of the metric $\mathsf m$ on $\Bbb CP^1$.  For the explicit expressions of $\tau$ in terms of  basic objects on the Riemann surface (prime form, theta-functions, etc.) and the divisor of the meromorphic differential  $df$, we refer the interested readers to~\cite{KK1,KS} for genus zero  and one surfaces,  and to~\cite{KK,KKZ} if $g>1$.

As a simple illustrating example consider  the map $f:\Bbb CP^1_w\to \Bbb CP^1_z$, where 
$$
z=f(w):=\frac {z_2w^2+z_1(z_2-z_1)}{w^2+z_2-z_1}.
$$ 
The function $f$ has simple critical points at $w=0$ and $w=\infty$, the corresponding critical values are $z_1$ and $z_2$. 
We have  $\tau(z_1,z_2)=\sqrt[4]{z_1-z_2}$  (see~\cite[f-la (3.38)]{KK1}) and~\eqref{teorema} takes the form
\begin{equation}\label{simple}
\Det'\Delta_F=C \sqrt{|z_1-z_2|}\sqrt[8]{\rho(z_1,\bar z_1)\cdot \rho(z_2,\bar z_2)}. 
\end{equation}
Let, for instance, $\mathsf m$ be the hyperbolic metric on $\Bbb CP^1_z$ with three conical points at  $z=p_j$ of angles $2\pi(\beta_j+1)$ satisfying   $\beta_1+\beta_2+\beta_3<-2 $ and  $\beta_j\in(-1,0)$.  Then the  hyperbolic metric $f^*\mathsf m$ on $\Bbb CP^1_w$ has conical singularities of angles  $2\pi(\beta_j+1)$ at 
$$
w=\pm\sqrt{(z_1-z_2)\frac{z_1-p_j}{z_2-p_j}}, \quad j=1,2,3, 
$$
 and two conical points of angles $4\pi$ at $w=0$ and $w=\infty$. Substituting the corresponding explicit formula~\cite{KRS} for $\rho(z,\bar z)$  into~\eqref{simple}  we obtain an explicit formula for the determinant of $\Delta_F$ on   $(\Bbb CP^1_w, f^*\mathsf m)$ with a  constant  $C$ that does not depend on the parameters $z_1$ and $z_2$ in $\Bbb C P_z\setminus\{p_1,p_2,p_3\}$, $z_1\neq z_2$. 

Let us also notice that  by setting $\rho(z,\bar z)=4(1+|z|^2)^{-2}$ in~\eqref{teorema} (i.e. by  taking the standard curvature $1$ metric as the metric $\mathsf m$ on the Riemann sphere $\Bbb CP^1$) we immediately obtain the main result of the recent paper~\cite{KKIMRN}. 

The plan of this paper is as follows. In Section~\ref{Psec2} entitled ``Preliminaries''  we   
list some  known results on asymptotic expansions of solutions near conical singularities at the critical points $P_k$ of $f$.  In Section~\ref{Psec3}, which is the core of the present paper,  we find an explicit formula for a coefficient that is responsible for the metric dependent factor $\prod_{k=1}^M\sqrt[8]{\rho(z_k,\bar z_k)}$ in  the explicit formula~\eqref{teorema} for $\Det'\Delta_F$.  In Section~\ref{Psec4} we study dependance of the eigenvalues of  $\Delta_F$ on the moduli parameters $z_1,\dots, z_M$. Finally, in Section~\ref{Psec5} we first define the modified zeta regularized determinant ${\Det}'\Delta_F$ of the Friedrichs Laplacian on $(X,f^*\mathsf m)$ and then prove the  formula~\eqref{teorema}.

\section{ Preliminaries}\label{Psec2}
 
Let $\Delta_F$ stand for the Friedrichs selfadjoint extension of the Laplace-Beltrami operator $\Delta$  on $(X,f^*\mathsf m)$. In this subsection we list some results on asymptotic expansions at  the critical points of $f$ for functions in the domain of $\Delta_F$ .
 These results  are  similar to those that previously appeared in the context of flat conical metrics~\cite{HK,HKK} and  standard curvature $1$ metric on $\Bbb CP^1$~\cite{KKIMRN}. We only state the results, for the proofs we refer to~\cite[Sec. 2]{KKIMRN} as they can be repeated here almost verbatim.  
 
In a vicinity of  a critical point $P_k$, $k=1,\dots,M$, of $f:X\to \Bbb CP^1$ we introduce  the {\it distinguished} local parameter 
\begin{equation}\label{LocPar}
x=\sqrt{f(P)-f(P_k)}=\sqrt{z-z_k}.
\end{equation}
Since the critical value $z_k=f(P_k)$ of $f$ does not coincide with any conical singularity $p_j$ of $\mathsf m$, we have $f^*\mathsf m({x},\bar x)=4 \hat \rho(x,\bar x )|x|^2|dx|^2$ with smooth near $x=0$ function $\hat \rho(x,\bar x)= \rho(x^2+z_k,
\bar x^2+\bar z_k)$, where $\rho(z,
\bar z)$ is the conformal factor of $\mathsf m$.  Thus $f^*\mathsf m$ has a conical singularity of angle $4\pi$  at $P_k$  (i.e. at $x=0$).

\begin{lemma}\label{expansion} Denote by $\mathscr D^*$ the domain of the operator in $L^2(X)$ adjoint  to the operator $\Delta$ defined on  $C^\infty_0(X\setminus\{P_1,\dots,P_k\})$. 
For $u\in \mathcal D^*$ in a small vicinity of ${x}=0$ we have
\begin{equation}\label{EU}
u({x},\bar {x})=a_{-1}\bar {x}^{-1} +b_{-1}{x}^{-1}+  a_0\ln|{x}|+ b_0+ a_1\bar {x} +b_1 {x}+ R({x},\bar {x}),
\end{equation}
where  $a_k$ and $b_k$ are some coefficients and the remainder $R$ satisfies $R({x},\bar {x})=O(|{x}|^{2-\epsilon})$ with  any $\epsilon>0$  as $x\to 0$.  Moreover,  the equality~\eqref{EU}  can be differentiated and   the remainder satisfies $\partial_{x} R({x},\bar {x})=O(|{x}|^{1-\epsilon})$ and $\partial_{\bar {x}} R({x},\bar {x})=O(|{x}|^{1-\epsilon})$ with  any $\epsilon>0$ if $\Delta^* u=\lambda u$.
\end{lemma}

For $u,v\in \mathscr D^*$   consider the form $
\mathsf q[u,v]:=(\Delta u,v)-(u,\Delta v)
$;
here and elsewhere $(\cdot,\cdot)$ stands for the inner product in ${L^2(X)}$.  By Lemma~\ref{expansion} we have~\eqref{EU} and
\begin{equation}\label{Ev}
v({x},\bar {x})=c_{-1}\bar {x}^{-1} +d_{-1}{x}^{-1}+  c_0\ln|{x}|+ d_0+ c_1\bar {x} +d_1 {x}+ \tilde R({x},\bar {x}).
\end{equation}
The Stokes theorem  implies $$
\mathsf q[u,v]=2i\lim_{\epsilon\to 0+}\oint_{|{x}|=\epsilon}(\partial_{{x}}u)\bar v\,d {{x}}+u(\partial_{\bar {x}} \bar v)d\bar {x},
$$
where we express the right hand side  in terms of the coefficients in~\eqref{EU},~\eqref{Ev} and obtain
\begin{equation}\label{q}
\mathsf q[u,v]=4\pi(-a_{-1}\bar d_1-b_{-1}\bar  c_1-b_0\bar c_0/2+ a_0\bar d_0/2+b_1\bar c_{-1}+a_1\bar d_{-1}).
\end{equation}

 For the domain $\mathscr D$ of the Friedrichs extension $\Delta_F$ we have $\mathscr D=\mathscr D^*\cap H^1(X)$, where $H^1(X)$ is the domain of the closed densely defined quadratic form of $\Delta$. Hence for any $u\in\mathscr D$ we have~\eqref{EU} with $a_{-1}=b_{-1}=a_{0}=0$.

For a sufficiently small $\delta>0$ we take a cut-off function $\chi\in C^\infty_c(X)$  supported in the neighbourhood $|x|<2\delta$ of $P_k$ and  such that $\chi(|x|)=1$ for $|x|<\delta$.  Denote the spectrum of $\Delta_F$ by $\sigma(\Delta_F)$ and introduce
\begin{equation}\label{Y}
Y(\lambda)=\chi x^{-1} -(\Delta_F-\lambda)^{-1}(\Delta-\lambda)\chi x^{-1}, \quad \lambda\notin \sigma(\Delta_F),
\end{equation}
where the function $\chi x^{-1}$ is extended from the support of $\chi$ to $X$ by zero.
It is clear that $Y(\lambda)\in \mathscr D^*$ and $Y(\lambda)\not =0$ as $\chi x^{-1}\notin\mathscr D$. 
By  Lemma~\ref{expansion}  we  have
\begin{equation}\label{expY}
Y(x,\bar x; \lambda)=x^{-1}+c(\lambda)+a(\lambda)\bar x+b(\lambda) x+ O(|x|^{2-\epsilon}),\quad x\to0,\quad \epsilon>0.
\end{equation}

\begin{lemma}\label{b(lambda)}  The function $Y(\lambda)$ in~\eqref{Y} and the coefficient $b(\lambda)$ in~\eqref{expY}  are analytic functions of $\lambda$ in $\Bbb C\setminus\sigma(\Delta_F)$ and in a neighbourhood of zero. Besides, we have
 \begin{equation}\label{refl}
 4\pi\frac{d}{d\lambda}b(\lambda)=\bigl(Y(\lambda),\overline{Y(\lambda)}\bigr).
 \end{equation}
 \end{lemma}
\begin{lemma}\label{Eigenf} Let $\{\Phi_j\}_{j=0}^\infty$ be a complete set of real normalized eigenfunctions of $\Delta_F$ and let $\{\lambda_j\}_{j=0}^\infty$ be the corresponding eigenvalues, i.e. $\Delta_F\Phi_j=\lambda_j\Phi_j$, $\Phi_j=\overline{\Phi_j}$, and $\|\Phi_j; L^2(X)\|=1$. Then for the coefficients $a_j$ and  $b_j=\bar a_j $  in the asymptotic
\begin{equation}\label{Phi_j}
\Phi_j(x,\bar x)=c_j+ a_j\bar x+ b_j x+O(|x|^{2-\epsilon}),\quad x\to 0, \quad \epsilon>0,
\end{equation}
we have
\begin{equation}\label{a_j}
16\pi^2\sum_{j=0}^\infty \frac{b_j^2}{(\lambda_j-\lambda)^2}=\Bigl(Y(\lambda),\overline{Y(\lambda)}\Bigr),
\end{equation}
where  the series is absolutely convergent.
\end{lemma}

\section{Explicit formula for  $b(-\infty)$}\label{Psec3}
In this section we first study the behaviour of the coefficient $b(\lambda)$ from \eqref{expY}  as $\lambda\to-\infty$ and obtain an explicit formula for the limit $b(-\infty)=\lim_{\lambda\to-\infty}b(\lambda)$. This is the core of the problem.  For singular  flat and round metrics the corresponding explicit formulas  were derived with the help of separation of variables~\cite{HK,HKK,KKIMRN}.  In our setting this approach does not work and we use a different idea:  In a neighbourhood of $x=0$  that shrinks to  zero as $\lambda\to-\infty$ we construct  an approximation for $Y(\lambda)$ that is sufficiently good to find $b(-\infty)$.

It is also important to notice that an expression for $b(-\infty)$ needs to be integrated (i.e. written as the derivative of some function with respect to  $z_k$) before it can be used in explicit formulas for the determinant  $\Det'\Delta_F$. This can be not an easy task even if an explicit expression for  $b(-\infty)$ was found  by separation of variables~\cite{Torus 2017} (e.g. separation of variables can be used to find $b(-\infty)$ if  $\mathsf m$ is a constant curvature singular metric)  and hence one may want to use the general integrated explicit expression given in Lemma~\ref{binfty} below.

\begin{lemma} \label{binfty} As $\lambda\to-\infty$ for the coefficient $b(\lambda)$ in~\eqref{expY} we have
\begin{equation*}
b(-\infty)=\partial_{z_k}\ln( \rho(z_k,\bar z_k))^{-1/4},\quad b(\lambda)=b(-\infty)+O(|\lambda|^{-1/2}),
\end{equation*}
where $ \rho(z_k,\bar z_k)$ is the value of the conformal factor of the metric $\mathsf m(z,\bar z)=\rho (z,\bar z) |dz|^2$  at the critical value $z_k$ of $f: X\to\Bbb CP^1$. 
\end{lemma}
\begin{proof} Consider a metric plane $(\Bbb C, \mathsf m_0)$ isometric  to an euclidean cone of total angle $4\pi$. Let $\mathsf m_0(x,\bar x)=4\hat \rho_0|x|^2|dx|^2$ with $\hat\rho_0:=\rho(z_k,\bar z_k)$, i.e.  $\mathsf m_0(x,\bar x)$ is the metric $f^*\mathsf m(x,\bar x)=4\hat\rho(x,\bar x) |x|^2|dx|^2$  with conformal factor $\hat\rho(x,\bar x)$ frozen  at $x=0$.   Let
$$
Y_0(k,x,\bar x)=\frac {1}{x}e^{-kx\bar x}, \text{ where } k=\sqrt{-\hat\rho_0\lambda},\quad  0>\lambda\to-\infty.
$$ Clearly  $Y_0(x,\bar x; \lambda)=x^{-1}+O(1)$ as $x\to 0$  (cf.~\eqref{expY}) and $ \Delta_0 Y_0=\lambda Y_0$,
where $$\Delta_0=-\frac{1}{\hat\rho_0}\frac{1}{|x|^2}\partial_x\partial_{\bar x}$$ is the Laplacian  on $(\Bbb C, \mathsf m_0)$. In the remaining part of the proof we first show that  $Y_0$ is a sufficiently good approximation of $Y$   and then use $Y_0$ to calculate $b(-\infty)$.

Let $\tilde\chi_{j}\in C^\infty(\overline{\Bbb R^+})$, $j=1,2$, be  cut-off functions such that $\tilde\chi_j(t)=1$ for $t\leq j$ and $\tilde\chi_j(t)=0$ on $t\geq 3j/2$. For $x\in \Bbb C$ set $\chi_j(k,x,\bar x):=\tilde\chi_j(k^{1/4}|x|)$, where $k>0$ is a parameter. Clearly  $\chi_1\chi_2=\chi_1$ and the support of $\chi_{j}$ shrinks to $x=0$ as $k\to+\infty$ (i.e. as $\lambda\to -\infty$).

 For all sufficiently large $k>0$ we extend the functions $\chi_{j} Y_0$,  $(\Delta_0-\lambda)\chi_{j} Y_0=[\Delta_0,\chi_{j}] Y_0$ (here $[a,b]=ab-ba$), and $(\Delta-\Delta_0)\chi_{j}Y_0$ (considered as functions of the distinguished local parameter~\eqref{LocPar}) from their supports in a small vicinity of $x=0$ to $X$ by zero. From the explicit formula for $Y_0$ we immediately see that  
\begin{equation}\label{key1}
\|[\Delta_0,\chi_{j}] Y_0; L_2(X)\|=O(|\lambda|^{-\infty})\quad \text{ as }  \lambda\to-\infty. 
\end{equation}
Hence
$$
\begin{aligned}
(\Delta-\Delta_0)\chi_{j}Y_0 & =\bigl({\hat\rho_0}/{\hat\rho}-1\bigr)\Delta_0\chi_{j}Y_0\\& =\bigl({\hat\rho_0}/{\hat\rho}-1\bigr)(\chi_{j}\Delta_0Y_0 +[\Delta_0,\chi_{j}]Y_0)
\\
&=\bigl({\hat\rho_0}/{\hat\rho}-1\bigr)\chi_{j}\lambda Y_0+O(|\lambda|^{-\infty})\in L^2(X).
\end{aligned}
$$
On the next step we estimate $\|\bigl({\hat\rho_0}/{\hat\rho}-1\bigr)\chi_{j} \lambda Y_0; L^2(X)\|$.

On the support of $\chi_j$ we have $k^{1/4}|x|\leq  3j/2$. This together with
 $$
 {\hat \rho_0}/{\hat \rho(x,\bar x)}=1-\frac{\rho_{z_k}(z_k,\bar z_k)}{\hat \rho_0} x^2-\frac{\rho_{\bar z_k}(z_k,\bar z_k)}{\hat \rho_0} \bar x^2 +O(|x|^4)  
 $$
implies 
$$
 {\hat \rho_0}/{\hat \rho(x,\bar x)}-1=-\frac{\rho_{z_k}(z_k,\bar z_k)}{\hat \rho_0} x^2-\frac{\rho_{\bar z_k}(z_k,\bar z_k)}{\hat \rho_0} \bar x^2 +O(k^{-1}),\quad x\in\supp \chi_j.
 $$
After the change of variables $(x,\bar x)\mapsto (k^{-1/2}x,k^{-1/2}\bar x )$ in 
$$
\begin{aligned}
\|\bigl({\hat\rho_0}/{\hat\rho}-1\bigr)&\chi_{j} \lambda Y_0; L^2(X)\|^2\\&=\int_{\Bbb C} \bigl({\hat\rho_0}/{\hat\rho}(x,\bar x)-1\bigr)^2\tilde\chi_{j}^2(k^{1/4}| x|) |\lambda|^2 \frac{1}{|x|^2}e^{-2k|x|^2}\hat\rho(x,\bar x) 4|x|^2\frac{dx\wedge d\bar x}{-2i}
\end{aligned}
$$
we obtain
$$
\|\bigl({\hat\rho_0}/{\hat\rho}-1\bigr)\chi_{j} \lambda Y_0; L^2(X)\|^2\leq  C\int_{\Bbb C} k^{-2} |\lambda|^2e^{-2|x|^2}k^{-1}\frac{dx\wedge d\bar x}{-2i}=O(k)=O(|\lambda|^{1/2}).
$$
We demonstrated that 
\begin{equation}\label{key2}
\|(\Delta-\Delta_0)\chi_{j}  Y_0;L^2(X)\|=O(|\lambda|^{1/4})\text{ as } \lambda\to-\infty.
\end{equation}

Now we represent the function  $Y$ from~\eqref{Y} in the form 
\begin{equation}\label{tilde}
 Y=\chi_j Y_0+ ( \Delta_F -\lambda)^{-1}(\Delta-\lambda)\chi_j Y_0.
 \end{equation}
Since $$(\Delta-\lambda)\chi_j Y_0=(\Delta-\Delta_0)\chi_{j}  Y_0+[\Delta_0,\chi_j]Y_0,$$
 the estimates~\eqref{key1} and \eqref{key2} together with  $\| ( \Delta_F -\lambda)^{-1}; \mathcal B(L^2(X))\|=O(|\lambda|^{-1})$ and~\eqref{tilde} imply 
\begin{equation}\label{key3}
 \|Y-\chi_j Y_0; L^2(X)\|=O(|\lambda|^{-3/4}); 
 \end{equation}
 in this sense $Y_0$ is a good approximation of $Y$.

Let $b(\lambda)$ be the coefficient from the asymptotic \eqref{expY} of $Y$. 
We have
\begin{equation}\label{star1}
\begin{aligned}
4\pi b(\lambda)& =\mathsf q\left[Y-\chi_1 Y_0, \overline{Y}\right]=\left((\Delta-\lambda)(Y-\chi_1Y_0),\overline{Y}\right)=\left(-(\Delta-\lambda)\chi_1Y_0,\overline{Y}\right)\\
&=-\Bigl([\Delta_0,\chi_1] Y_0+(\Delta-\Delta_0)\chi_1 Y_0,
\\ &\qquad\qquad\qquad\overline{\chi_2  Y_0+ ( \Delta_F -\lambda)^{-1}([\Delta_0,\chi_{2}]Y_0+(\Delta-\Delta_0)\chi_{2}Y_0)}\Bigr).
\end{aligned}
\end{equation}
Thanks to~\eqref{key1}, \eqref{key2}, and \eqref{key3}, the right hand side of~\eqref{star1} takes the form
$$
-\left((\Delta-\Delta_0)\chi_1 Y_0,\overline{\chi_2  Y_0}\right)+O(|\lambda|^{-1/2})=-\Bigr(\bigl({\hat\rho_0}/{\hat\rho}-1\bigr)\chi_{1}\lambda Y_0, \overline{Y_0}\Bigr)+O(|\lambda|^{-1/2}).
$$
It remains to calculate the value of the inner product in the right hand side. We have
$$
-\Bigr(\bigl({\hat\rho_0}/{\hat\rho}-1\bigr)\chi_{1}\lambda Y_0, \overline{Y_0}\Bigr)=-\int_{\Bbb C}\Bigl(\hat\rho_0-\hat\rho(x,\bar x)\Bigr)\tilde\chi_1(k^{1/4}|x|)\lambda\frac{1}{x^2}e^{-2k|x|^2} 4 |x|^2\frac{dx\wedge d\bar x}{-2i}.
$$
Here we substitute $x=k^{-1/2}re^{i\phi}$ (recall that $\lambda=-k^2/\hat\rho_0$) and obtain
$$
\begin{aligned}
4 & \int_0^\infty\int_0^{2\pi} \Bigl(1-\hat\rho(k^{-1/2}r,\phi)/\hat\rho_0\Bigr)\tilde\chi_1(k^{-1/4}r)k e^{-2i\phi}e^{-2 r^2} r \,d\phi \,dr 
\\
& = -\frac{4}{\hat \rho_0}\int_0^\infty\int_0^{2\pi} \Bigl(\rho_{z_k}(z_k,\bar z_k)\frac{ r^2}{k}e^{2i\phi}+\rho_{\bar z_k}(z_k,\bar z_k)\frac{r^2}{k}e^{-2i\phi}+ O\Bigl(\frac {|r|^4}{k^2}\Bigr)\Bigr)
\\
& \qquad\qquad\qquad\qquad\qquad\qquad\qquad\qquad\qquad \times\tilde\chi_1(k^{-1/4}r)k e^{-2i\phi}e^{-2 r^2} r \,d\phi \,dr 
\\
& = -\frac{4}{\hat \rho_0}\int_0^\infty\Bigl(2\pi\rho_{z_k}(z_k,\bar z_k){ r^2}+ O\Bigl(\frac {|r|^4}{k}\Bigr)\Bigr)\tilde\chi_1(k^{-1/4}r) e^{-2 r^2} r  \,dr 
\\
& =-\frac{ 8\pi \rho_{z_k}(z_k,\bar z_k)}{\hat \rho_0}\int_0^\infty e^{-2 r^2} r^3  \,dr +O(k^{-1}) =-\pi \frac{\rho_{z_k}(z_k,\bar z_k)}{\rho(z_k,\bar z_k)}+O(|\lambda|^{-1/2}).
\end{aligned}
$$
\end{proof}

\section{Variation of eigenvalues under perturbation of conical singularities } \label{Psec4}

Pick a noncritical value  $z_0\in \Bbb C$ of $f$ such that the critical values $z_1,\dots,z_M$ of $f$ are the end points but not internal points of the line segments  $[z_0,z_k]$, $k=1,\dots,M$. The complement $X\setminus f^{-1}(\mathsf U)$ of the preimage  $f^{-1}(\mathsf U)$  of the unit $\mathsf U=\cup_{k=1}^M[z_0,z_k]$ in $X$ has $N$
connected components ($N$ sheets of the covering) and $f$ is a biholomorphic isometry from each of these components equipped with  metric $f^*\mathsf m$ to  $\Bbb CP^1\setminus \mathsf U$ equipped with metric $\mathsf m$.
 Thus the  Riemann manifold $(X, f^*\mathsf m)$ is isometric to the one obtained by  gluing $N$ copies of the Riemann sphere $(\Bbb CP^1,\mathsf m)$ along the cuts $\mathsf U$ in accordance with a certain gluing scheme.
Perturbation of the conical singularity at $P_k$  is  a small  shift by $w\in\Bbb C$, $|w|<<1$, of the end $z_k$ of the cut $[z_0,z_k]$ to $z_k+w$ on those two copies of the Riemann sphere $(\Bbb CP^1,\mathsf m)$ that  produce $4\pi$-conical angle at $P_k$ after gluing along  $[z_0,z_k]$.

Let $\varrho\in C^\infty_0(\Bbb R)$ be a cut-off function  such that $\varrho(r)=1$ for $x<\epsilon$ and $\varrho(r)=0$ for $r>2\epsilon$, where $\epsilon$ is small. Consider the selfdiffeomorphism
$$
\phi_w(z,\bar z)=z+\varrho(|z-z_k|)w
$$
of the Riemann sphere  $\Bbb CP^1$, where $w\in \Bbb C$ and $|w|$ is small. On those  two copies of the Riemann sphere  that   produce the conical singularity at $P_k$ after  gluing along  $[z_0,z_k]$,  we shift $z_k$ to $z_k+w$ by applying  $\phi_w$. 
We assume that the support of $\varrho$ and the value $|w|$ are so small that only $[z_0,z_k]$ and no other cuts are affected by $\phi_w$. Let  $(X,f_w^*\mathsf m)$ stand for the perturbed manifold, where $f_w: X\to \Bbb CP^1$ is the  meromorphic function with  critical values $z_1,\dots,z_{k-1}, z_k+w,z_{k+1},\dots,z_M$.
 Consider $(X,f_w^*\mathsf m)$ as  $N$ copies of the Riemann sphere $\Bbb CP^1$ glued along the (unperturbed) cuts $\mathsf U$, however  $N-2$ copies are endowed with metric $\mathsf m$ and $2$ certain  copies (mutually glued along $[z_0,z_k]$) are endowed with pullback $\phi_w^*\mathsf m$ of $\mathsf m$ by $\phi_w$. By $\Delta_w$ we denote the Friedrichs extension of Laplace-Beltrami operator on $(X,f_w^*\mathsf m)$ and consider $\Delta_w$ as a perturbation of $\Delta_0$ on $(X,f^*\mathsf m)$.

In a vicinity of $P_k$ for the matrix of $f^*_w\mathsf m $ we have $$[f_w^*\mathsf m](x,\bar x)=\rho\circ\phi_w(x^2-z_k,\bar x^2-\bar z_k) \left(\phi'_w(x,\bar x)\right)^*{\phi'_w(x,\bar x)}, $$
where $x$ is the local distinguished parameter~\eqref{LocPar}  and 
$$
\phi'_w(x,\bar x)=\Id+\frac{\varrho'(|x|^2)}{2|x|^2}\left[
\begin{array}{cc}
   w \bar x^2  & w x^2  \\

  \bar w\bar x^2   & \bar wx^2
\end{array}
\right]
$$
is the Jacobian matrix;  i.e. $
f_w^*\mathsf m=2[\bar xd\bar  x \  x d x][f_w^*\mathsf m] [xdx\ \bar xd\bar x]^T.
$

Notice that $f^*\mathsf m\neq f_w^*\mathsf m$ only in a small neighborhood of $P_k$, where $|x|\leq \sqrt{2\epsilon}$, and $f_w^*\mathsf m$ is not in the conformal class of $f^*\mathsf m$ only for $\sqrt{\epsilon}\leq|x|\leq  \sqrt{2\epsilon}$, where $\varrho'(|x|^2)\neq0$. The norms in the spaces $L^2(X,f^*\mathsf m)$ and $L^2(X,f_w^*\mathsf m)$ are equivalent and the spaces can be identified. 

A direct computation shows that 
\begin{equation}\label{D-D}
\begin{aligned}
\Delta_w -\Delta_0=&\left\{-\frac{\varrho(|x|^2)}{\hat\rho(x,\bar x)}\left(\frac{\hat\rho_x(x,\bar x)}{2x}w+ \frac{\hat\rho_{\bar x}(x,\bar x)}{2\bar x}\bar w\right) -\frac{\varrho'(|x|^2)}{2|x|^2}\bigl(\bar x^2w+ x^2\bar w\bigr)\right\}\Delta_{0} \\
&+\frac{1}{2|x|^2\hat\rho(x,\bar x)}\bigl(\partial_x\varrho'(|x|^2)w\partial_x +\partial_{\bar x} \varrho'(|x|^2)\bar w\partial_{\bar x}\bigr)+O(|w|^2),
\end{aligned}
\end{equation}
where $\hat\rho(x,\bar x)=\rho(x^2+z_k,\bar x^2+\bar z_k)$ as before and $O(|w|^2)$ stands for a second order operator with smooth coefficients  supported in $|x|\leq \sqrt{2\epsilon}$ and uniformly bounded by $C|w|^2$.
Now it is  easy to see that not the domain $\mathscr D^*_w$ of the operator adjoint to $\Delta_w\restriction_{C_0(X\setminus\{P_1,\dots,P_k\})}$  nor the  Sobolev space $H^1(X, f_w^*\mathsf m)$ (the domain of the closure of  quadratic form of $\Delta_w\restriction_{C_0(X\setminus\{P_1,\dots,P_k\})}$) depend on $w$. Hence  the domain $\mathscr D=\mathscr D^*_w\cap H^1(X, f_w^*\mathsf m)$ of the Friedrichs extension $\Delta_w$ does not depend on $w$ either. 
From now on we consider $\Delta_w$ as an  operator in  $L^2(X):=L^2(X,f^*\mathsf m)$ with domain $\mathscr D$.

 Let $\Gamma$ be a closed curve enclosing an eigenvalue $\lambda$ of $\Delta_0$ of multiplicity $m$ and  not any other eigenvalues of $\Delta_0$. The resolvent $(\Delta_w-\xi)^{-1}$ exists for all $\xi\in\Gamma$ provided $|w|$ is sufficiently small. Moreover,  as $|w|\to 0$ the difference  $(\Delta_w-\xi)^{-1}-(\Delta_0-\xi)^{-1}$ tends to zero in the norm of  $\mathcal B(L^2(X),\mathscr D)$ uniformly in $\xi\in\Gamma$; here  $\mathscr D$ is eqiped with the  graph norm of $\Delta_0$. 
 The continuity of the total projection 
 $$
P_w=-\frac{1}{2\pi i}\oint_\Gamma (\Delta_w-\xi)^{-1}\,d\xi
$$
implies that $\dim P_w L^2(X)=\dim P_0 L^2(X)=m$; i.e. the sum of multiplicities of the eigenvalues of $\Delta_w$ lying inside $\Gamma$ is equal to $m$ (provided $|w|$ is small). These eigenvalues are said to form a $\lambda$-group, see e.g.~\cite{Kato}.  
 
 \begin{lemma}\label{diff}  Consider a $\lambda$-group $\lambda_1(w),\dots,\lambda_m(w)$, here $\lambda_j(w)\to \lambda_j=\lambda$ as $w\to0$.  Let $\Phi_1,\dots \Phi_m$ be (real) normalized eigenfunctions of $\Delta_0$ corresponding to $\lambda$, i.e. $\Phi_j=\overline \Phi_j$, $\|\Phi_j; L^2(X)\|=1$, and $\operatorname{span}\{\Phi_1,\dots \Phi_m\}=P_0L^2(X)$. Then as $w\to 0$  we have
\begin{equation}\label{LGroup}
\sum_{j=1}^m\frac{1}{\bigl(\xi-\lambda_j(w)\bigr)^2}=\frac{m}{(\xi-\lambda)^2}+\frac{2(Aw+B\bar w )}{\bigl(\xi-\lambda\bigr)^3}+O(|w|^2),
\end{equation}
where $A=\sum_{j=1}^m A_j$ and $B=\sum_{j=1}^m B_j$ with coefficients  $A_j$ and $B_j$ from the expansion
\begin{equation}\label{CoeffAB}
\Bigl((\Delta_w-\Delta_0)\Phi_j,\Phi_j\Bigr)_{L^2(X)}=A_j w+B_j\bar w+O(|w|^2). 
\end{equation}

\end{lemma}
\begin{proof} For the proof we refer to~\cite[Lemmas 5.1---5.3 ]{KKIMRN}.  In the proof of~\cite[Lemma 5.1]{KKIMRN} the formula~\eqref{CoeffAB} was obtained  (see\cite[f-la (28)]{KKIMRN}) and then the coefficients $A$ and $B$ were calculated for curvature $1$ conical metrics. That calculation does not work for metrics we study here and therefore should be omitted. With formulas for coefficients replaced by~\eqref{CoeffAB}   Lemmas  5.1---5.3 in~\cite{KKIMRN}  and their proofs can be repeated here verbatim. The assertion of this Lemma~\ref{diff}  is then the one of ~\cite[Lemma 5.3]{KKIMRN}. 
\end{proof}
 
 As a consequence of~\eqref{D-D} for the coefficients $A_j$ and $B_j$ in~\eqref{CoeffAB} we obtain
$$
A_j=-\int_{|x|\leq\sqrt{2\epsilon}}\left[ \Bigl(2\bar x\varrho(|x|^2)\hat\rho_x +2\bar x^2\varrho'(|x|^2)\hat\rho\Bigr) \lambda\Phi_j^2
+2\varrho'(|x|^2)\Bigl(\partial_x\Phi_j\Bigr)^2
 \right]\frac{dx\wedge d\bar x}{-2i},
$$
$$
B_j=-\int_{|x|\leq\sqrt{2\epsilon}}\left[ \Bigl(2x\varrho(|x|^2)\hat\rho_{\bar x}+2x^2\varrho'(|x|^2)\hat\rho\Bigr) \lambda\Phi_j^2
+2\varrho'(|x|^2)\Bigl(\partial_{\bar x}\Phi_j\Bigr)^2
 \right]\frac{dx\wedge d\bar x}{-2i}.
$$
This together with Stokes theorem gives
$$
A_j=i\lim_{\varepsilon\to0+}\oint_{|{x}|=\epsilon}\frac{1}{{x}}(\partial_{x}\Phi_j)^2\,d{x}-\lambda\Phi_j^2\hat \rho \bar x\, d\bar x,
$$
$$ 
B_j=-i\lim_{\varepsilon\to0+}\oint_{|{x}|=\epsilon}\frac{1}{\bar {x}}(\partial_{\bar {x}}\Phi_j)^2\,d\bar {x} -\lambda\Phi_j^2\hat \rho  x\, d x.
$$
Now we use the asymptotic~\eqref{Phi_j} of $\Phi_j(x,\bar x)$  to conclude that  $\Phi_j^2(x,\bar x)\hat \rho(x,\bar x) $ is bounded near $x=0$. By Lemma~\ref{expansion} the asymptotic can be differentiated. Thus for any $\epsilon>0$  we have $\partial_x\Phi_j=b_j+O(|x|^{1-\epsilon})$ as $x\to 0$.  We arrive at
 \begin{equation}\label{A_B}\begin{aligned}
A_j=i\lim_{\varepsilon\to0+}\oint_{|{x}|=\epsilon}\frac{1}{{x}}(\partial_{x}\Phi_j)^2\,d{x}=2\pi b_j^2,\\
B_j=-i\lim_{\varepsilon\to0+}\oint_{|{x}|=\epsilon}\frac{1}{\bar {x}}(\partial_{\bar {x}}\Phi_j)^2\,d\bar {x} =2\pi \bar b_j^2.
\end{aligned}
 \end{equation} 

\section{ Explicit formula for  $\Det' \Delta_F$}\label{Psec5}

In this subsection we first define the  modified (i.e. with zero eigenvalue excluded) zeta regularized determinant $\Det' \Delta_F$ and then prove the explicit formula~\eqref{teorema}. 

By~\cite[Theorem 4.1]{Troyanov Polar Coordinates}, near any conical singularity  of  $f^*\mathsf m$ there exist smooth local  polar geodesic coordinates $(r,\theta)$ such that   $$f^*\mathsf m=dr^2+h^2(r,\theta)d\theta^2,$$ where  
$$
\lim_{r\to0}\frac{h(r,\theta)}{r}=1,\quad  h_r(0,\theta)=1,\quad  h_{rr}(0,\theta)=0
$$
for  $\theta\in [0,\gamma)$; here $ h_r=\partial_r h$, $h_{rr}=\partial_r^2 h$, and $\gamma=2\pi(\beta_j+1)$ (resp. $\gamma=4\pi$)  if  the point $r=0$ is a preimage $f^{-1}(p_j)$ of conical singularity of $\mathsf m$ (resp.  a critical point of $f$). 

Let $\chi\in C^\infty(X)$ be a cutoff function supported in a small  neighbourhood $r<\epsilon$ of the conical singularity  and such that $\chi=1$ for $r<\epsilon/2$. Then $\chi\Delta_F$ can be  considered as the operator 
$$
\chi h^{-1/2}\left(-\partial_r^2+r^{-2}\mathcal A(r)\right)h^{1/2}
$$
acting in $L^2(h(r,\theta)\,dr \,d \theta)$, where 
$$
[0,\epsilon] \ni r\mapsto \mathcal A(r)=-r^2\left(\frac{ h_r^2}{4h^{2}}-\frac{ h_{rr}} {2 h}  +h^{1/2}\left(\frac{1}{h}\partial_\theta\right)^2h^{-1/2}\right)
$$
is a smooth family of operators on the circle $\Bbb R/\gamma\Bbb Z$. Therefore the results of~\cite{BS2} are applicable.  By~\cite[Thm~5.2 and Thm 7.1]{BS2} we have 
\begin{equation}\label{HA}
\Tr \chi e^{-\Delta_F t}\thicksim\sum_{j=0}^\infty A_j t^{\frac {j-3}{2}} +\sum_{j=0}^\infty B_j t^{-\frac{\alpha_j+4}{2}}+\sum_{j: \alpha_j\in\Bbb Z_- } C_j  t^{-\frac{\alpha_j+4}{2}} \log t\quad \text{    as } t\to0+
\end{equation}
with some coefficients $A_j$, $B_j$, and $C_j$,  and a sequence of complex numbers  $\{\alpha_j\}$,  $\Re\alpha_j\to-\infty$. The coefficient $C_j$ before $t^0\log t$ in~\eqref{HA} is given by
$\frac1 4\Res \zeta(-1)$, where $\zeta$ is the $\zeta$-function of $(\mathcal A(0)+1/4)^{1/2}$; see \cite[f-la (7.24)]{BS2}. Since $\mathcal A(0)=-\partial^2_\theta-1/4$,  we have
$$\zeta(s)=2\sum_{j\geq1} (\pi j/\gamma)^{-s}=2 (\gamma/\pi)^{s}\zeta_R(s),$$
where $\zeta_R$ stands for the Riemann zeta function. Thus $\Res \zeta(-1)=0$ and  the coefficient $C_j$ before  $t^0\log t$  is  zero.

For any cut-off function $\chi\in C^\infty(X)$ supported outside of conical singularities of $f^*\mathsf m $  the short time asymptotic  $\Tr \chi e^{-\Delta_F t}\sim \sum_{j\geq -2} a_jt^{j/2} $ can be obtained in the standard well known way, see e.g.~\cite[Problem 5.1]{Shubin} or~\cite{Seeley}.  

In summary, there is no term $C_j t^0\log t$ in the short time asymptotic of $\Tr e^{-\Delta_F t}$ and hence the $\zeta$-function $$\zeta(s)=\frac{1}{\Gamma(s)}\int_0^\infty t^{s-1}(\Tr e^{-t\Delta_F} -1)\,dt $$ has no pole at zero. Now we are in position to define  the modified determinant $$\Det' \Delta_F: =\exp\{-\zeta'(0)\}.$$

Chose oriented paths  $a_1,\dots, a_g$ and $b_1,\dots, b_g$ marking the Riemann surface $X$ of genus $g$. Let $\{v_1, \dots,
v_g\}$ be the normalized basis of holomorphic differentials on $X$, i.e. 
$\int_{a_\ell}v_m=\delta_{\ell m}$, 
where $\delta_{\ell m}$ is the Kronecker delta. Introduce the $g\times g$-matrix $\mathbb B=(\mathbb B_{\ell m})$ of $b$-periods with entries ${\mathbb B}_{\ell m}=\int_{b_\ell}v_m$. The following theorem is the main result of this paper.

\begin{theorem}\label{MAINTHM}  Let $\mathsf m$ stand for a conical metric on the Riemann sphere $\Bbb CP^1$ such that the curvature of $\mathsf m$ is smooth and the volume is finite.   Consider  a meromorphic function $f: X\to {\Bbb  C}P^1$  of degree $N$  with simple poles and $M=2N+2g-2$ simple critical points $P_k$. Assume that the critical values $z_k=f(P_k)$ of $f$ are finite and do not coincide with conical singularities of $\mathsf m$.  Consider the determinant ${\Det}'\Delta_F$ of the Friedrichs extension  $\Delta_F$ of the Laplacian on  $(X,f^*{\mathsf m})$ as a functional on the Hurwitz moduli space $H_{g,N}(1,\dots,1)$ of pairs $(X,f)$ with local coordinates $z_1,\dots, z_M$. Then the    explicit formula
$$
{\Det}'\Delta_F=C\,{\det}\Im {\mathbb B}\, |\tau|^2 \prod_{k=1}^M\sqrt[8]{\rho(z_k,\bar z_k)}
$$
holds true, where $C$ is independent of the point $(X,f)$  of   $H_{g,N}(1,\dots,1)$, $\Bbb B$ is the matrix of $b-$periods on $X$, and $\tau$ is the Bergman $\tau$-function on $H_{g,N}(1,\dots,1)$.  (The factor ${\det}\Im {\mathbb B}$ should be omitted if   $X$ is  a genus zero surface.) 
\end{theorem}
\begin{proof} The proof runs along the lines of ~\cite[Proposition 6.1 and Theorems 6.2 and 6.3]{KKIMRN} with some minor changes due to replacement of the formula $b(-\infty)=\frac{1}{2}\frac{\bar z_k}{1+|z_k|^2}$ (valid in the case of standard round metric $\mathsf m$ on $\Bbb CP^1$) by  the general formula obtained in Lemma~\ref{binfty}. Notice also that the final results of Section~\ref{Psec4}, i.e. f-las~\eqref{LGroup} and~\eqref{A_B},   turn out to be the same as in the case of standard round metric, thus requiring no  changes in the proof. For these reasons we only outline the main steps of the proof below  and refer to~\cite[Proposition 6.1 and Theorems 6.2, 6.3]{KKIMRN} for  details.

First we  compute the partial derivative of zeta function with respect to $z_k$. Let $\Gamma_\lambda$ be a contour running clockwise at a sufficiently small distance $\epsilon>0$ around the cut $(-\infty, \lambda]$. Thanks to  Lemma~\ref{diff} and~\eqref{A_B} we obtain
$$\partial_{z_k} \zeta(s;\Delta_F-\lambda) =\frac{1}{2\pi i (s-1)}\int_{\Gamma_\lambda}(\xi-\lambda)^{1-s}\partial_{z_k}\Tr(\Delta_F-\xi)^{-2}\,d\xi
$$
$$
=\frac{2i}{  (s-1)}\int_{\Gamma_\lambda}(\xi-\lambda)^{1-s} \sum_{j=0}^\infty \frac{b_j^2}{\bigl(\lambda_j-\xi\bigr )^3}\,d\xi.
$$
We integrate in the right hand side by parts and then use the equality~\eqref{a_j} from Lemma~\ref{Eigenf} together with Lemma~\ref{b(lambda)} and the estimate from Lemma~\ref{binfty}. We  get
$$\begin{aligned}
\partial_{z_k} \zeta(s;\Delta_F-\lambda) & =\frac{-i}{16\pi^2}\int_{\Gamma_\lambda}(\xi-\lambda)^{-s}\bigl(Y(\xi),\overline{Y(\xi)}\bigr)\,d\xi\
\\
&=\frac{-i}{4\pi}\int_{\Gamma_\lambda}(\xi-\lambda)^{-s}
\frac{d}{d\xi}\left\{b(\xi)-b(-\infty)\right\}\,d\xi\\&=\frac{-is}{4\pi}\int_{\Gamma_\lambda}(\xi-\lambda)^{-s-1}\left\{b(\xi)-b(-\infty)\right\}\,d\xi.
\end{aligned}
$$
Since $\xi\mapsto b(\xi)$ is holomorphic in $\Bbb C\setminus\sigma(\Delta_F)$ and in a neighbourhood of zero (Lemma~\ref{b(lambda)}), the Cauchy Theorem implies
$$
\partial_{z_k} \zeta'(0;\Delta_F) =\frac{1}{4\pi i}\int_{\Gamma_\lambda}(\xi-\lambda)^{-1}\{b(\xi)-b(-\infty)\}\,d\xi=\frac{b(-\infty)-b(0)}{2}\,.
$$
 The coefficient  $b(0)$ is conformally invariant ($Y(0)$ is a harmonic function bounded everywhere  on $X$ except for the point $P_k$ ; see~\eqref{Y},~\eqref{expY},  and Lemma~\ref{b(lambda)}) and thus the equality
$$
b(0)=\partial_{z_k} \ln ({\det} \Im {\mathbb B}\,|\tau|^2)$$
 can be obtained in exactly the same way as in~\cite[Lemma 4.2]{KKIMRN} or~\cite[Prop. 6]{HKK}.  Since $\Det'\Delta=\exp\{-\zeta'(0)\}$,  this together with  formula for  $b(-\infty)$ (see Lemma~\ref{binfty}) gives
\begin{equation}\label{System}
\partial_{z_k} \ln {\Det}'\Delta_F=\partial_{z_k} \ln ({\det} \Im {\mathbb B}\,|\tau|^2)-\partial_{z_k}\ln( \rho(z_k,\bar z_k))^{-1/8}, \quad k=1, \dots, M.
\end{equation}
 Similarly,  
$\partial_{\bar z_k} \ln {\Det}'\Delta_F$ is given by the conjugate of the right hand side in~\eqref{System}. This completes the proof. 

\end{proof}

{\bf Acknowledgements.} 
It is a pleasure to thank  Alexey Kokotov for helpful discussions and important remarks.

\end{document}